\documentclass[a4paper,preprint,showkeys]{amsproc}
\usepackage{mathrsfs}
\usepackage{amsmath,booktabs}
\usepackage[T4, OT1]{fontenc}
\usepackage[arrow, matrix, curve]{xy}
\usepackage{newunicodechar}
\usepackage{graphicx}
\usepackage{bm}
\usepackage[dvips]{epsfig}
\usepackage{rotating}
\usepackage{amsmath,amssymb,amsfonts,amsthm}
\usepackage{amsmath,amscd}
\usepackage{subfig}
\usepackage{caption}
\usepackage{pst-all}
\usepackage{dcolumn}
\usepackage[thinc]{esdiff}
\usepackage{amsthm}
\usepackage{dsfont}
\usepackage{amssymb}
\usepackage[hyphens]{url} 
\usepackage[dvips]{graphicx}
\usepackage{mathrsfs}
\makeatletter
\@namedef{subjclassname@2020}{\textup{2020} Mathematics Subject Classification}
\makeatother
\theoremstyle{plain}

 \newtheorem{lem}{Lemma}[section]
 \newtheorem{cor}{Corollary}[section]
\theoremstyle{definition}
 \newtheorem{exm}{Example}[section]
 
\theoremstyle{remark}
 \newtheorem{rem}{Remark}[section]
 \numberwithin{equation}{section}

\renewcommand{\geq}{\geqslant}

\title[A study of generalized Dirichlet Integrals.]{A short study of generalized Dirichlet Integrals\\ via tempered distributions.}

\author[Belardinelli]{Cyril Belardinelli} 
\address{ 
Kantonsschule Solothurn\\
Solothurn\\
Switzerland}
\email{cyril.belardinelli@ksso.ch}
\begin{document}
\vspace{18mm} \setcounter{page}{1} \thispagestyle{empty}
\begin{abstract}
In the present article, the author uses Fourier theory of tempered distributions (generalized functions) in deriving a formula for Dirichlet-like integrals. The applied method is remarkably efficient and allows a solution in a few calculational steps.
The interest of the present article lies not only in the derivation of an appearently unknown formula but also in the original methodology of its derivation. 
\end{abstract}
\maketitle
\section{\label{histo}Introduction}
In the present article we are concerned with integrals of the following type:\\
\begin{equation}
\label{formula_0}
I(m,n):=\int_{0}^{\infty}\frac{\sin^{n}{x}}{x^m}\mathrm{d}x  
\end{equation}
where $n$, $m$ are positive integers.\\    
The improper integral $I(m,n)$ is absolut convergent when $n\geq m\geq 2$.\\
The case $m=1$ has to be treated separately since the integrals $I(1,n)$ are not absolut convergent. It turns out that 
these integrals are divergent for even $n$ and convergent for odd $n$.
A corresponding formula for the case of equal exponents $n=m$ is known \cite{apostol:1980,remmert:2001}. The simplest case $n=m=1$ is the well-known Dirichlet Integral:\\
\begin{equation}
\label{formula1}
I(1,1)=\int_{0}^{\infty}\frac{\sin{x}}{x}\mathrm{d}x=\frac{1}{2}\pi
\end{equation}
\\ 
\newpage
The main purpose of this article is to derive the following formulae.\\
\\
In the case $n\geq m\geq 2$:
\begin{equation}
\label{formula_-1}
I(m,n)=\frac{i^{m-n+1}}{2^{n}(m-1)!}\sideset{}{'}\sum^{n}_{l=0}(-1)^{l}\binom{n}{l}(n-2l)^{m-1}\ln[{i(2l-n)}]
\end{equation}
In the case $m=1$:
\begin{equation}
\label{formel_Zwei}
I(1,n)=\begin{cases}
\text{divergent}&\quad n\,\text{even}\\
\frac{-1}{(2i)^n}\sum^{n}_{l=0}(-1)^{l}\binom{n}{l}\ln[{i(2l-n)}]&\quad n\,\text{odd}\\
\end{cases}
\end{equation}
where in formula \ref{formula_-1} the term $2l=n$ is omitted in the summation. Note that the complex logarithm in the above formulae may be written as follows:
\begin{equation*}
\ln[{i(2l-n)}]=\ln{|2l-n|}+i\frac{\pi}{2}\,\text{sign}(2l-n)
\end{equation*}
\begin{exm}
\begin{equation*}
\begin{split}
I(1,1)&=\frac{1}{2}\pi \quad\quad I(1,2)=\text{divergent}\qquad\qquad\qquad\qquad\,\, I(1,3)=\frac{1}{4}\pi\\
I(2,2)&=\frac{1}{2}\pi \quad\quad I(2,3)=\frac{3}{4}\ln{3}\qquad\qquad\qquad\qquad\qquad I(2,4)=\frac{1}{4}\pi \\
I(3,3)&=\frac{3}{8}\pi \quad\quad I(3,4)=\ln{2}\qquad\qquad\qquad\qquad\qquad\,\,\  I(3,5)=\frac{5}{32}\pi \\
I(4,4)&=\frac{1}{3}\pi \quad\quad I(4,5)=\frac{125}{96}\ln{5}-\frac{45}{32}\ln{3}\qquad\qquad\,\,\,  I(4,6)=\frac{1}{8}\pi \\
I(5,5)&=\frac{115}{384}\pi\quad I(5,6)=\frac{27}{16}\ln{3}-2\ln{2}\qquad\qquad\quad\,\,\, \, I(5,7)=\frac{77}{768}\pi \\
\\
\end{split}
\end{equation*}
\end{exm}
\subsection{Dirichlet Integral}
In the present section we introduce the method by calculating the well-known Dirichlet Integral. We rewrite this integral in the following suggestive form:
\begin{equation*}
\int_{0}^{\infty}\frac{\sin{kx}}{x}\mathrm{d}x=\frac{1}{2i}\int_{-\infty}^{\infty}\frac{\Theta(x)\,\mathrm{e}^{-i(-k)x}}{x}-\frac{\Theta(x)\,\mathrm{e}^{-ikx}}{x}\mathrm{d}x
\end{equation*}
Where $\Theta(\cdot)$ denotes Heaviside's unit step function defined by:
\begin{equation*}
\Theta(x)=\begin{cases}
1&\text{if} \quad x>0\\
\frac{1}{2}&\text{if} \quad x=0\\
0&\text{if} \quad x<0
\end{cases}
\end{equation*}
Obviously, splitting of the latter integral is not possible. However, if one reinterprets 
the quantity $\frac{\Theta(x)}{x}$ as a tempered distribution rather than a function (See Appendix A) one may
write:
\begin{equation}
\label{zusatz}
  \int_{0}^{\infty}\frac{\sin{kx}}{x}\mathrm{d}x=\frac{1}{2i}\mathscr{F}\left[\frac{\Theta(x)}{x}\right]_{-k}-\frac{1}{2i}\mathscr{F}\left[\frac{\Theta(x)}{x}\right]_{k}
\end{equation}
Note that the Fourier transform $\mathscr{F}\left[\frac{\Theta(x)}{x}\right]$ is well-defined in the context of tempered distributions. The explicit formula reads as follows (See Appendix \ref{appendix_1}):
\begin{equation}
\label{formula_10}
\mathscr{F}\left[\frac{\Theta(x)}{x}\right]=-\gamma-\ln{|k|}-i\frac{\pi}{2}\text{sign}(k)=-\gamma-\ln{ik}
\end{equation}
where $\gamma=0.5772156\dots$ denotes Euler's constant.
By inserting formula \ref{formula_10} into Eq.~\ref{zusatz} we obtain the correct result:
\begin{equation*}
\label{formula2}
\int_{0}^{\infty}\frac{\sin{kx}}{x}\mathrm{d}x=\frac{\pi}{2}\,\text{sign}(k)
\end{equation*}
\section{\label{Sec.II}Derivation of the general formula}
In the present section we derive a general formula for the integrals \ref{formula_0}. For this purpose we need the following Lemma.
\begin{lem}
\label{lemma1}
\begin{equation*}
\left(\frac{\Theta}{x}\right)^{(n)}=-\delta^{(n)}(x)\left[1+\frac{1}{2}+\frac{1}{3}+\dots+\frac{1}{n}\right]
+(-1)^n\,n!\,\frac{\Theta(x)}{x^{n+1}}
\end{equation*}
where $\delta^{(n)}(x)$ denotes the $nth$ derivative of Delta-Dirac-distribution.
\end{lem}
\begin{proof}
We give a proof by induction on n.
Evidently, the Lemma holds for $n=1$:
\begin{equation*}
\left(\frac{\Theta}{x}\right)^{\prime}
=\delta(x)\cdot\frac{1}{x}
-\frac{\Theta(x)}{x^2}=-\delta^{\prime}(x)-\frac{\Theta(x)}{x^2}
\end{equation*}
Where the following well-known identities have been used.
\begin{equation*}
\begin{split}
&\Theta^{\prime}(x)=\delta(x)\\
&\delta^{\prime}(x)=-\frac{\delta(x)}{x}
\end{split}
\end{equation*}
Inductive step $n\Rightarrow n+1$:
\begin{equation*}
\begin{split}
\left(\frac{\Theta}{x}\right)^{(n+1)}&=-\delta^{(n+1)}(x)\left[1+\frac{1}{2}+\frac{1}{3}+\dots+\frac{1}{n}\right]
+(-1)^n\,n!\,\left[\frac{\delta(x)}{x^{n+1}}-(n+1)\frac{\Theta(x)}{x^{n+2}}\right]\\
&=-\delta^{(n+1)}(x)\left[1+\frac{1}{2}+\frac{1}{3}+\dots+\frac{1}{n+1}\right]+
(-1)^{n+1}\,(n+1)!\,\,\frac{\Theta(x)}{x^{n+2}}
\end{split}
\end{equation*}
Where the following equally well-known formula for the $nth$ derivative of delta function has been used:
\begin{equation*}
\delta^{(n)}(x)=(-1)^n\,n!\,\frac{\delta(x)}{x^n}
\end{equation*}
\end{proof}
Lemma \ref{lemma1} leads to the following Corollary:
\begin{cor}
\label{corollary_1}
\begin{equation}
\label{formula_3}
\mathscr{F}\left[\frac{\Theta(x)}{x^{n+1}}\right]=\frac{(-ix)^n}{n!}\,\left(\mathscr{F}\left[\frac{\Theta(x)}{x}\right]+\sum^{n}_{l=1}\frac{1}{l}\right)
\end{equation}
\end{cor}
\begin{proof}
On one hand, $\Theta(x)/x$ fulfills the following identity:
\begin{equation}
\label{gleichung1}
\mathscr{F}\left[\left(\frac{\Theta(x)}{x}\right)^{(n)}\right]=(ix)^{n}\mathscr{F}\left[\frac{\Theta(x)}{x}\right]
\end{equation}
On the other hand, according to Lemma \ref{lemma1}, we have:
\begin{equation}
\label{gleichung2}
\mathscr{F}\left[\left(\frac{\Theta(x)}{x}\right)^{(n)}\right]=-(ix)^{n}\sum^{n}_{l=1}\frac{1}{l}+(-1)^n\,n!\,\mathscr{F}\left[\frac{\Theta(x)}{x^{n+1}}\right]
\end{equation}
We obtain formula of Corollary \ref{formula_3} by equating Eqs.~\ref{gleichung1}, \ref{gleichung2}.
\end{proof}
It is now a straightforward task to derive a formula for the integrals \ref{formula_0}. By virtue of the Eulerian formula we may use:
\begin{equation*}
\sin^{n}{x}=\left(\frac{\mathrm{e}^{ix}-\mathrm{e}^{ix}}{2i}\right)^n
\end{equation*}
Expanding the latter expression yields:
\begin{equation}
\label{expansion}
\sin^{n}{x}=\frac{1}{(2i)^n}\sum^{n}_{l=0}\binom{n}{l}(-1)^l \mathrm{e}^{-ix(2l-n)}
\end{equation}
By inserting expansion \ref{expansion} into \ref{formula_0} and by adding an exponential we obtain formally:
\begin{equation}
\label{formula_6}
\int_{0}^{\infty}\frac{\sin^{n}{x}}{x^m}\mathrm{e}^{-\epsilon x}\mathrm{d}x=
\frac{1}{(2i)^n}\sum^{n}_{l=0}\binom{n}{l}(-1)^l\mathscr{F}\left[\frac{\Theta(x)}{x^m}\right]_{k=2l-n-i\epsilon}
\end{equation}
We now take advantage of Corollary \ref{formula_3} and formula \ref{formula_10} which leads to:
\\
\begin{equation}
\label{formula_20}
\begin{split}
\int_{0}^{\infty}\frac{\sin^{n}{x}}{x^m}\mathrm{e}^{-\epsilon x}\mathrm{d}x&=\frac{i^{m-n+1}}{2^{n}(m-1)!}\sum^{n}_{l=0}(-1)^{l}\binom{n}{l}(n-2l+i\epsilon)^{m-1}\ln[{i(2l-n-i\epsilon)}]\\
&+(C-\gamma)\frac{i^{m-1-n}}{2^n\,(m-1)!}\sum^{n}_{l=0}(-1)^l\,\binom{n}{l}\,(n-2l+i\epsilon)^{m-1}
\end{split}
\end{equation}
\\
Where $C:=1+\frac{1}{2}+\frac{1}{3}+\dots+\frac{1}{m-1}$
\\
\\
In the next step we take the limit $\epsilon\to 0$ on both sides of Eq.~\ref{formula_20}.\\
\\
In the case $n\geq m>1$ we obtain the formula:
\begin{equation*}
\int_{0}^{\infty} \frac{\sin^{n}x}{x^m}\mathrm{d}x=\frac{i^{m-n+1}}{2^{n}(m-1)!} \sideset{}{'}\sum^{n}_{l=0}(-1)^{l}\binom{n}{l}(n-2l)^{m-1}\ln[{i(2l-n)}]
\end{equation*}
Where the following identity has been used (See Appendix B)
\begin{equation*}
\sum^{n}_{l=0}(-1)^l\,\binom{n}{l}\,(n-2l)^{m-1}=0,\quad n\geq m\geq 2
\end{equation*}
In the case $m=1$ we get:
\\
\begin{equation*}
\int_{0}^{\infty}\frac{\sin^{n}x}{x}\mathrm{d}x=\begin{cases}
\text{divergent}&\quad n\,\text{even}\\
-\frac{1}{(2i)^n}\sum^{n}_{l=0}(-1)^{l}\binom{n}{l}\ln[{i(2l-n)}]&\quad n\,\text{odd}\\
\end{cases}
\end{equation*}
\\
where the following elementary identity has been used:
\begin{equation*}
\sum^{n}_{l=0}(-1)^l\,\binom{n}{l}=(1-1)^n=0
\end{equation*}
\begin{rem}
By using the identity
\begin{equation*}
\ln[{i(2l-n)}]=\ln{|2l-n|}+i\frac{\pi}{2}\,\text{sign}(2l-n)
\end{equation*}
one derives easily the following formula for $n\geq m>1$:
\\
\begin{equation*}
\int_{0}^{\infty}\frac{\sin^{n}x}{x^m}\mathrm{d}x = \left\{
\begin{array}{lll}
\frac{i^{m-n+2}}{2^{n}(m-1)!}\frac{\pi}{2}\sideset{}{'}\sum^{n}_{l=0}(-1)^{l}\binom{n}{l}(n-2l)^{m-1}\,\text{sign}(2l-n) &\textrm{if $n-m$ even} \\
& \\
\frac{i^{m-n+1}}{2^{n}(m-1)!}\sideset{}{'}\sum^{n}_{l=0}(-1)^{l}\binom{n}{l}(n-2l)^{m-1}\ln{|2l-n|} & \, \textrm{if $n-m$ odd} \\
\end{array}
\right. 
\end{equation*}
\\
A formula without case distinction is possible by using the following definitions:
\begin{equation*}
\begin{split}
0^{m-1}\ln{(i\cdot0)}&:=0\quad \quad\ m=2, 3,\dots\\
0^{m-1}\ln{(i\cdot0)}&:=+\infty\quad    m=1\\
\end{split}
\end{equation*}
Presuming this we may write down a formula valid in the entire range $n\geq m\geq1$:
\begin{equation*}
I(m,n)=\frac{i^{m-n+1}}{2^{n}(m-1)!}\sum^{n}_{l=0}(-1)^{l}\binom{n}{l}(n-2l)^{m-1}\ln[{i(2l-n)}]
\end{equation*}
Evidently, these definitions get their justification due to the following limits:
\begin{equation*}
\lim_{x\to 0} x^{m-1}\ln{ix}=\begin{cases}
0&\text{if} \quad m=2, 3, \dots\\
\infty&\text{if} \quad m=1\\
\end{cases}
\end{equation*}
\end{rem}
\appendix
\section{}
\label{appendix_1}
We rewrite the function $\Theta(x)/x$ in the form: 
\begin{equation*}
\frac{\Theta(x)}{x}=\frac{1}{2|x|}+\frac{1}{2x}
\end{equation*}
It is well-known that the functions $1/|x|$ and $1/x$ can be turned into tempered distributions via the following operations \cite{schwartz:1950,schwartz:1951,gelfand:1964}:
\begin{equation*}
\begin{split}
\mathrm{PV}\frac{1}{x}&:=\lim_{\epsilon\to0^{+}}\int\limits_{|x|>\epsilon}\frac{\varphi(x)}{x}\mathrm{d}x\\
\mathrm{Pf}\frac{1}{|x|}&:=\mathrm{f.\,p.\,of}\lim_{\epsilon\to0^{+}}\int\limits_{|x|>\epsilon}\frac{\varphi(x)}{|x|}\mathrm{d}x
\end{split}
\end{equation*}
Where \emph{PV} refers to Cauchy's Principal Value and \emph{Pf} to Hadamard's finite part. 
The function $\varphi$ denotes a test function in the Schwartz space $\mathscr{S}(\mathbb{R})$ of smooth and rapidly decaying functions.\\
As a tempered distribution $\Theta(x)/x$ has a well-defined Fourier transform. In order to derive it we take advantage of the following formula \cite{gelfand:1964}:
\begin{equation}
\label{formel}
\mathscr{F}\left[\frac{1}{|x|}\right]=-2\gamma-2\ln{|x|}
\end{equation}
By inserting the identity 
\begin{equation*}
\frac{1}{|x|}=\frac{2\,\Theta(x)-1}{x}
\end{equation*}
into formula \eqref{formel}
we obtain
\begin{equation*}
\mathscr{F}\left[\frac{\Theta(x)}{x}\right]=-\gamma-\ln{|x|}-i\frac{\pi}{2}\text{sign}(x)=-\gamma-\ln{ix}
\end{equation*}
\section{}
\begin{equation*}
\sum^{n}_{l=0}(-1)^l\,\binom{n}{l}\,(n-2l)^{m-1}=0,\quad n\geq m\geq 2
\end{equation*}
\begin{proof}
It is enough to prove the following identity:
\begin{equation*}
\sum^{n}_{l=0}(-1)^l\,\binom{n}{l}\,l^{m-1}=0,\quad n\geq m\geq 2
\end{equation*}
We proceed by induction on $m$. 
As shown in the following, the identity holds for $m=2$ and arbitrary $n\geq m$. By using the binomial theorem we obtain:
\begin{equation*}
\diff{}{x}(1-x)^n=-n\,(1-x)^{n-1}=\sum^{n}_{l=0}(-1)^l\,\binom{n}{l}\,l\,x^{l-1},\quad x\neq 0
\end{equation*}
Inserting the value $x=1$ leads to the statement for $m=2$.\\
\\
Inductive step $m\Rightarrow m+1$:
\begin{equation}
\label{app_B_1}
\diff[m+1]{}{x}(1-x)^n=(-1)^{m+1}\,n\,(n-1)\cdots (n-m)\,(1-x)^{n-m-1} 
\end{equation}
On the other hand:
\begin{equation}
\label{app_B_2}
\diff[m+1]{}{x}(1-x)^n=\sum^{n}_{l=0}(-1)^l\,\binom{n}{l}\,\underbrace{l\,(l-1)\cdots(l-m)}_{=l^{m+1}+P_{m}(l)}\,x^{l-m-1},\quad x\neq 0
\end{equation}
Where $P_{m}(l)$ denotes a polynom in $l$ of order $m$. Once again, equating Eq.~\ref{app_B_1} and Eq.~\ref{app_B_2} and inserting the value $x=1$ leads to the statement.
\end{proof}
\appendix
\bibliography{References_paper_dirichlet_integral_general}
\bibliographystyle{plain}
\end{document}